\documentclass[12pt]{amsart}
\usepackage{amsmath,amssymb,amsfonts,graphics}

\newtheorem{thm}{Theorem}[section]
\newtheorem{cor}[thm]{Corollary}
\newtheorem{lem}[thm]{Lemma}
\newtheorem{prop}[thm]{Proposition}
\theoremstyle{definition}
\newtheorem{defn}[thm]{Definition}
\theoremstyle{remark}
\newtheorem{exm}[thm]{Example}
\newtheorem{rem}[thm]{Remark}
\numberwithin{equation}{section}

\newcommand{\R}{\mathbb{R}}
\newcommand{\N}{\mathbb{N}}
\newcommand{\Z}{\mathbb{Z}}
\newcommand{\T}{\mathbb{T}}
\newcommand{\C}{\mathbb{C}}
\newcommand{\Cm}{\mathbb{C}^*}
\newcommand{\Rm}{\mathbb{R}_+}

\newcommand{\fL}{\mathcal{L}}

\newcommand{\Zc}{\mathcal{Z}}
\newcommand{\Hc}{\mathcal{H}}
\newcommand{\Nc}{\mathcal{N}}
\newcommand{\tw}{\circledast}
\newcommand{\tc}{\circledast_\T}
\newcommand{\Sm}{\mathcal{S}}
\newcommand{\Zb}{\Zc_{b}^2(G,\Cm)}
\newcommand{\Zbs}{\Zc_{bs}^2(G,\Cm)}
\newcommand{\Zbw}{\Zc_{bw}^2(G,\Cm)}
\newcommand{\ZTb}{\Zc_{b}^2(G,\mathbb{T})}

\newcommand{\om}{\omega}
\newcommand{\Om}{\Omega}
\newcommand{\sg} {\sigma}

\newcommand{\GT}{\widetilde{G}}

\newcommand{\supp}{\text{supp}}

\begin{document}

\title[Twisted Orlicz algebras, I]
{Twisted Orlicz algebras, I}

\author[Serap \"{O}ztop]{Serap \"{O}ztop}
\address{Department of Mathematics, Faculty of Science, Istanbul University, Istanbul, Turkey}
\email{oztops@istanbul.edu.tr}

\author{Ebrahim Samei}
\address{Department of Mathematics and Statistics, University of Saskatchewan, Saskatoon, Saskatchewan, S7N 5E6, Canada}
\email{samei@math.usask.ca}

\footnote{{\it Date}: \today.



The second named author was supported by NSERC Grant no. 409364-2015 and
2221-Fellowship Program For Visiting Scientists And Scientists On Sabbatical Leave
from Tubitak, Turkey.}




\maketitle

\begin{abstract}
Let G be a locally compact group, let $\Om:G\times G\to \Cm$ be a 2-cocycle, and let $\Phi$ be a Young function. In this paper, we consider the Orlicz space $L^\Phi(G)$ and investigate
its algebraic property under the twisted convolution $\tw$ coming from $\Om$.
We find sufficient conditions under which $(L^\Phi(G),\tw)$ becomes a Banach algebra or
a Banach $*$-algebra; we call it a {\it twisted Orlicz algebra}. Furthermore, we study its harmonic analysis properties, such as symmetry, existence of functional calculus, regularity, and having Wiener property, mostly for the case when $G$ is a compactly generated group of polynomial growth.
We apply our methods to several important classes of polynomial as well as subexponential weights and demonstrate that our results could be applied to variety of cases.
\end{abstract}


In harmonic analysis, an important object related to a locally compact group $G$ and its (left) Haar measure $ds$ is its group algebra $L^1(G):=L^1(G,ds)$. Since the Haar measure is invariant under the left translation,
the group actions on $G$ extend to $L^1(G)$ so that it becomes a Banach $*$-algebra with respect to the convolution 
$$ (f*g)(t)=\int_G f(s)g(s^{-1}t) ds,$$
and the involution
$$f^*(t)=\overline{f(t^{-1})}\Delta(t^{-1}),$$
where $\Delta$ is the modular function of $G$.
The properties of $L^1(G)$ has been well-studied over the last couple of decades
and one could deduce lots information about $G$ from $L^1(G)$ and vice versa. For instance, the unitary representations of $G$ are in one-to-one correspondence with
the non-degenerate bounded $*$-representations of $L^1(G)$. One could also consider
the ``twisted" convolution and involution on $L^1(G)$, i.e.
$$ (f\tw g)(t)=\int_G f(s)g(s^{-1}t) \Om_\T(s,s^{-1}t) ds,$$
and
$$f^*(t)=\overline{f(t^{-1})}\Delta(t^{-1})\overline{\Om_\T(t,t^{-1})},$$
where $\Om_\T$ is a 2-cocycle on $G$ with values in the unit circle $\T$. These concepts appear naturally when one consider the ``projective" unitary representations of $G$,
as well as, in other areas of mathematics such as non-commutative geometry or Gabor analysis (see, for example, \cite{EL}, \cite{GL}, \cite{M}).

A generalization of $L^1(G)$ is the $L^p(G)$ space for $1\leq p< \infty$ 
which is a Banach space, even a Banach module over $L^1(G)$ with respect to the convolution but it is a Banach algebra only when $G$ is compact \cite{S}. On the other hand, a very natural phenomenon occurring in harmonic analysis is the appearance of a ``weight" on the group or a ``weighted norm" on the algebra in computations. A weight $\om$ on $G$ is a locally bounded measurable function from $G$ into the positive reals. For such a weight, one can extend the construction of $L^p(G)$ to the ``weighted" $L_\om^p(G)$, i.e.
$$L_\om^p(G):=\{f: f\om \in L^p(G)\ \text {and}\ \|f\|_\om=\|f\om\|_p \}.$$
These spaces have various properties and numerous applications in harmonic analysis.
For instant, by applying the Fourier transform, we know that Sobolev spaces $W^{k,2}(\T)$ are nothing but certain weighted $l_\om^2(\Z)$ spaces.

A particular aspect of the behavior of weighted $L^p$ spaces over locally compact groups is that they could form an algebra with respect to the convolution! More precisely, when $p=1$ and $\om$ is submultiplicative, it follows routinely that $L_\om^1(G)$ is a Banach algebra. Even though, this may not hold in general if $p>1$, there are sufficient conditions under which $L_\om^p(G)$ is a Banach algebra with respect to the convolution.
This was first shown by J. Wermer for $G=\R$ in \cite{JW} and
Yu. N. Kuznetsova later extended it to general locally compact groups. She has also studied some important properties of $L_\om^p(G)$ as a Banach algebra such as the existence of an approximate identity and, for an abelian $G$, a description of their the maximal ideal space (see, \cite{K1}, \cite{K2}, and the references therein). Moreover, in \cite{KM} and together with C. Molitor-Braun, they studied other properties such as symmetry, existence of functional calculus and having the Wiener property.

In the mathematical analysis, Orlicz space is a type of function space which generalizes the $L^p$ spaces significantly. Besides the $L^p$ spaces, a variety of function spaces arises naturally in analysis in this way such as $L \log^+ L$ which is a Banach space related to Hardy-Littlewood maximal functions. They could also contain certain Sobolev spaces as subspaces. Similar to $L^p$ spaces,
one could also consider weighted Orlicz spaces and study their properties. Very recently, A. Osan\c{c}l{\i}ol  and S. \"{O}ztop have looked at weighted Orlicz spaces as Banach algebras with respect to convolution (\cite{OO}). They found sufficient conditions for which the corresponding space becomes an algebra and studied their properties such as existence of an approximate identity and the spectrum of the algebra when the underlying group is abelian. Their work, in part, extends some of the results of Kuznetsova to a wider classes of algebras.

Our goal in this paper is to continue studying convolution and possible algebraic structure on Orlicz spaces but in a more general setting that the one presented in \cite{OO}. In fact, we would like to have a theory that embeds everything we discuss above. We start with considering the twisted convolution coming from 2-cocycles with values in $\C^*$, the multiplicative group of complex numbers. We restrict ourselves to those 2-cocycles $\Om$ for which $|\Om|$ is a 2-coboundary determined by a submultiplicative weight $\om$. This approach has two advantages! First of all, it allows us to systematically and simultaneously study twisted convolution coming from 2-cocycles with values in $\T$ as well as the weighted spaces coming from $\om$. Secondly, the formulation we find on $\Om$ ensuring that the twisted Orlicz space becomes as algebra is that $|\Om|$ must satisfies a certain composition (see \eqref{Eq:2-cocycle bdd sum}) and this is much better presented and understood if the twisted convolution is viewed as the one coming from a 2-cocycle (Section \ref{S:Twisted Orlicz alg}). We call the algebras we obtain the {\it twisted Orlicz algebras}. When $\om$ is a symmetric weight, we show that there is a natural involutive structure on twisted Orlicz algebras over unimodular locally compact groups. We study their symmetry as Banach $*$-algebras.
and present a method to verify whether they are symmetric (Section \ref{S:symmetry}).

We apply our methods to study twisted Orlicz algebras over compactly generated groups of polynomial growth. We follow the line of work in \cite{KM} and study various harmonic analysis properties of these algebras such as symmetry, existence of functional calculus and having the Wiener property. We present three concrete and important classes of polynomial and subexponential weights on these groups and obtain a large family of
symmetric twisted Orlicz algebras. This demonstrate that our methods can be applied to vast variety of cases and extend the results even in the classical situation. For instance, it was left as open problems in \cite{KM} whether weighted $L^p$-algebras with subexponential weights over non-abelian groups are symmetric or have the Wiener property. We apply our methods and obtain affirmative answers in a much more general setting to these questions (Sections \ref{S:Symm-twisted Orlicz alg-PG groups}, \ref{S:subexp weights-symm Orlicz space}, and \ref{S:Functional calculus, Wiener property, and minimal ideals}).

We finish by pointing out that throughout this paper, we concern ourselves with the
theory ``bounded multiplications" for Banach algebras and Banach modules, as opposed to ``contractive multiplications". Also weights for us are ``weakly submultiplicative" as opposed to ``submultiplicative".  We have also investigated existence of approximate identities as well as cohomological properties of twisted Orlicz algebras where we will present them in a subsequent paper.


\section{Preliminaries}

In this section, we give some definitions and state some technical
results that will be crucial in the rest of this paper. In this
paper, $G$ denotes a locally compact group with a fixed left Haar
measure $ds$.

\subsection{Orlicz Spaces}

In this section, we recall some facts concerning Young functions and Orlicz spaces. Our main reference is \cite{rao}.

A nonzero function $\Phi:[0,\infty) \to[0,\infty]$ is called a Young
function if $\Phi$ is convex, $\Phi(0)=0$, and $\lim_{x\to \infty} \Phi(x)=\infty$. For a Young function $\Phi$, the complementary function $\Psi$ of
$\Phi$ is given by
\[\Psi(y)=\sup\{xy-\Phi(x):x\ge0\}\quad(y\geq 0).\]
It is easy to check that $\Psi$ is also a Young function. Also, if $\Psi$ is the complementary function of $\Phi$, then $\Phi$ is
the complementary of $\Psi$ and $(\Phi,\Psi)$ is called a
complementary pair. We have the Young inequality
\[xy\le\Phi(x)+\Psi(y)\quad(x,y\ge0)\]
for complementary functions $\Phi$ and $\Psi$. By our
definition, a Young function can have the value $\infty$ at a point,
and hence be discontinuous at such a point. However, we always consider the pair of complementary Young
functions $(\Phi,\Psi)$ with $\Phi$ being real-valued and continuous on $[0,\infty)$ and positive on $(0,\infty)$. Note that even though $\Phi$ is continuous,
it may happen that $\Psi$ is not continuous.

Now suppose that $G$ is a locally compact group with a fixed Haar measure $ds$ and $(\Phi,\Psi)$ is a complementary pair of Young functions. We define
\begin{align}\label{Eq:Orlicz defn-0}
\mathcal{L}^\Phi(G)=\left\{f:G\to\C:f \ \text{is measurable and}\  \int_G\Phi(|f(s)|)\,ds <\infty
\right\}.
\end{align}
Since $\mathcal{L}^\Phi(G)$ is not always a linear space, we define the the Orlicz space $L^\Phi(G)$ to be
\begin{align}\label{Eq:Orlicz defn}
L^\Phi(G)=\left\{f:G\to\C:\int_G\Phi(\alpha|f(s)|)\,ds <\infty
\mbox{ for some }\alpha>0\right\},
\end{align}
where $f$ indicates a member in the equivalence classes of measurable functions with respect to the Haar measure $ds$. Then the Orlicz space is a Banach space under the (Orlicz) norm $\|\cdot\|_\Phi$
defined for $f\in L^\Phi(G)$ by
\begin{align}\label{Eq:Orlicz norm}
\|f\|_\Phi=\sup\left\{\int_G|f(s)v(s)|\,ds: \int_G\Psi(|v(s)|)\,ds \le1\right\},
\end{align}
where $\Psi$ is the complementary function to $\Phi$. One can also
define the (Luxemburg) norm $N_\Phi(\cdot)$ on $L^\Phi(G)$ by
\begin{align}\label{Eq:Orlicz Luxemburg defn}
N_\Phi(f)=\inf\left\{k>0:\int_G\Phi\left(\frac{|f(s)|}{k}\right)
\,ds \le1\right\}.
\end{align}
It is known that these two norms are equivalent; that is,
\begin{align}\label{Eq:Orlicz norm-Luxemburg relation}
N_\Phi(\cdot)\le \|\cdot\|_\Phi\le2 N_\Phi(\cdot)
\end{align}
and
\begin{align}\label{Eq:Orlicz norm-defn relation}
N_\Phi(f)\le1 \ \ \text{if and only if}\ \ \int_G\Phi(|f(s)|)\,ds \le1.
\end{align}
Let $\Sm^\Phi(G)$ be the closure of the linear
space of all step functions in $L^\Phi(G)$. Then $\Sm^\Phi(G)$ is a Banach space
and contains $C_c(G)$, the space of all continuous functions on $G$ with compact support, as a dense subspace \cite[Proposition 3.4.3]{rao}. Moreover, $\Sm^\Phi(G)^*$, the dual of  $\Sm^\Phi(G)$, can be identified with $L^\Psi(G)$ in a natural way \cite[Theorem 4.1.6]{rao}. Another useful characterization of $\Sm^\Phi(G)$ is that
$f\in \Sm^\Phi(G)$ if and only if for every $\alpha>0$, $\alpha f\in \mathcal{L}^\Phi(G)$ \cite[Definition 3.4.2 and Proposition 3.4.3]{rao}.

A Young function $\Phi$ satisfies the $\Delta_2$ condition if there
exist a constant $K>0$ and $x_0\geq 0$ such that $\Phi(2x)\le K\Phi(x)$
for all $x\ge x_0$. In this case we write
$\Phi\in\Delta_2$.
If $\Phi\in\Delta_2$, then it follows that $L^\Phi(G)=\Sm^\Phi(G)$ so that  $L^\Phi(G)^*=L^\Psi(G)$ \cite[Corollary 3.4.5]{rao}. If, in addition, $\Psi\in\Delta_2$, then the Orlicz space $L^\Phi(G)$ is a reflexive Banach space.


We will frequently use the (generalized) H\"{o}lder's inequality for Orlicz spaces
\cite[Remark 3.3.1]{rao}.
More precisely, for any complementary pair of Young functions $(\Phi,\Psi)$
and any $f\in L^\Phi(G)$ and $g\in L^\Psi(G)$, we have
\begin{align}\label{Eq:Holder inequality}
\|fg\|_1:=\int_G |f(s)g(s)|ds \leq \min\{N_\Phi(f)\|g\|_\Psi , \|f\|_\Phi N_\Psi(g)\}.
\end{align}
This, in particular, implies that $fg\in L^1(G)$.

For $1\leq p<\infty$ and the Young function $\Phi(x)=\frac{x^p}{p}$,
the space $L^{\Phi}(G)$ becomes the Lebesgue space
$L^p(G)$ and the norm $\|\cdot\|_{\Phi}$ is equivalent to the
classical norm $\|\cdot\|_{p}$. If $p=1$, then the
complementary Young function of $\Phi(x)=x$ is
\begin{equation}\label{psisonsuz}
\Psi(y)=\left\{\begin{array}{ll}
0      & \text{if } 0\le y\le1,\\
\infty & \text{otherwise.}
\end{array}\right.
\end{equation}
and $\|f\|_{\Phi}=\|f\|_{1}$ for all $f\in L^1(G)$ since
$\int_G\Psi(|v(s)|)ds \leq 1$ if and only if $|v(s)|\le1$ locally almost
everywhere on $G$. Note that function $\Psi$ defined in
\eqref{psisonsuz} is still a Young function as in the first
definition of the Young function. If $1<p<\infty$, then the complementary
Young function of $\Phi(x)=\frac{x^p}{p}$ is $\Psi(y)=\frac{y^q}{q}$,
where $q$ is the conjugate of $p$, i.e. $\frac{1}{p}+\frac{1}{q}=1$.


\subsection{2-Cocycles and 2-Cobounaries}

Throughout this article, we use the notation $\Cm$ to denote the multiplicative group of complex numbers, i.e. $\Cm=\C \setminus \{0\}$,
$\Rm$ to be multiplicative group of positive real numbers, and $\T$ to be the unit circle in $\C$.

\begin{defn}\label{D:2-cocycle defn}
Let $G$ and $H$ be locally compact groups such that $H$ is abelian. A {\it (normalized) 2-cocycle on $G$ with values in $H$} is a Borel
measurable map $\Om: G\times G \to H$ such that
\begin{align}\label{Eq:2-cocycle relation}
\Om(r,s)\Om(rs,t)=\Om(s,t)\Om(r,st) \ \ \  (r,s,t \in G)
\end{align}
and
\begin{align}\label{Eq:2-cocycle relation normalization}
\Om(r,e_G)=\Om(e_G,r)=e_H \ \ \ (r\in G).
\end{align}
The set of all normalized 2-cocycles will be denoted by $\Zc^2(G,H)$.
\end{defn}
If $\om: G \to H$ is measurable with $\om(e_G)=e_H$, then it is easy to see that the mapping
$$(s,t)\mapsto \om(st)\om(s)^{-1}\om(t)^{-1}$$
satisfies \eqref{Eq:2-cocycle relation} and \eqref{Eq:2-cocycle relation normalization}. Hence it is a 2-cocycle; such maps are called 2-coboundaries. The set of 2-coboundaries will be denoted by $\Nc^2(G,H)$.
It is easy to check that $\Zc^2(G,H)$ is an abelian group under the product
$$\Om_1 \Om_2 (s,t)=\Om_1(s,t)\Om_2(s,t) \ \ (s,t\in G),$$
and $\Nc^2(G,H)$ is a (normal) subgroup of $\Zc^2(G,H)$. This, in particular, implies that
$$\Hc^2(G,H):=\Zc^2(G,H)/\Nc^2(G,H)$$ turns into a group; This is called the 2$^{nd}$ group cohomology of $G$ into $H$
with the trivial actions (i.e. $s\cdot \alpha=\alpha\cdot s=\alpha$ for all $s\in G$ and $\alpha \in H$).

We are mainly interested in the cases when $H$ is $\Cm$, $\Rm$ or $\T$. One essential observation is that
we can view $\Cm=\Rm \T$ as a (pointwise) direct product of groups. Hence, for any 2-cocycle
$\Om$ on $G$ with values in $\Cm$ and $s,t\in G$, we can (uniquely) write
$\Om(s,t)=|\Om(s,t)| e^{i\theta}$ for some $0\leq \theta <2\pi$. Therefore, if we put
\begin{align}
|\Om|(s,t):=|\Om(s,t)| \ \ \text{and}\ \ \Om_\T(s,t):=e^{i\theta},
\end{align}
then $\Om=|\Om|\Om_\T$ (in a unique way) and the mappings $|\Om|$ and $\Om_\T$ are 2-cocycles on $G$ with values in $\Rm$ and $\T$, respectively.


\section{Twisted group algebra}\label{S:Twisted group alg}

In this section, we present and summarize what we need from the theory of twisted group algebras. To start, we must restrict ourselves
to certain subgroups of 2-cocycles for which twisted group algebras can be defined. Throughout this section, $G$ is a locally compact group
with a fixed left Haar measure $ds$.


\begin{defn}\label{D:bound 2 cocycles}
We denote $\Zb$ to be the group of {\bf bounded 2-cocycles on $G$ with values in $\Cm$}
which consists of all elements $\Om\in \Zc^2(G,\Cm)$ satisfying the following conditions:\\
$(i)$ $\Om \in L^\infty(G\times G)$;\\
$(ii)$ $\Om_\T$ is continuous.\\
We also define $\Zbw$ to be the subgroup of $\Zb$ consisting of elements $\Om \in \Zb$ for which
$$|\Om|(s,t)=\frac{\om(st)}{\om(s)\om(t)} \ \ \ (s,t\in G),$$ where $\om : G \to \R_+$ is a locally integrable measurable function
with $\om(e)=1$ and $1/\om\in L^\infty(G)$. In this case, we call $\om$ a {\bf weight} on $G$ and say that $|\Om|$ is the {\bf 2-coboundary determined by} $\om$, or
alternatively, $\om$ is the {\bf weight associated to} $|\Om|$.
\end{defn}

Now suppose that $\Om\in \Zb$ and $f$ and $g$ are measurable functions on $G$. If there is a $\sigma$-finite measurable subset $E$ of $G$ such that
$f=g=0$ on $G\setminus E$, then we define the {\bf twisted convolution of $f$ and $g$ under $\Om$} to be
\begin{align}\label{Eq:twisted convolution}
(f\tw g) (t)=\int_G f(s)g(s^{-1}t)\Om(s,s^{-1}t)ds  \ \ \ (t\in G)
\end{align}
It follows from standard measure theory that, since $E\times E$ is $\sigma$-finite and its complement in $G\times G$ is a null set, $f\tw g$ is measurable on $G$.
Moreover, if we let $\Delta$ to be the modular function on $G$
and define $(s,t\in G)$
\begin{align}\label{Eq:twisted convolution-left right translation}
(\delta_s\tw f)(t)=f(s^{-1}t)\Om(s,s^{-1}t) \ \ \text{and}\ \ (f\tw \delta_s)(t)=f(ts^{-1})\Delta(s^{-1})\Om(ts^{-1},s),
\end{align}
then both functions $\delta_s\tw f$ and $f\tw \delta_s$ are measurable on $G$ and
\begin{align}\label{Eq:twisted convolution-operator translation}
(f\tw g) (t)=\int_G f(s)(\delta_s\tw g)(t)ds=\int_G g(s) (f \tw \delta_s)(t)ds  \ \ \ (t\in G).
\end{align}
It now follows routinely (using Fubini's theorem) that for every $f,g\in L^1(G)$,
$f\tw g\in L^1(G)$ with $\|f\tw g\|_1\leq \|\Om\|_\infty \|f\|_1\|g\|_1$.
We conclude that $(L^1(G),\tw)$ becomes a Banach algebra; it is called
the {\bf twisted group algebra}.

We would like to make some clarification with regard to our version of twisted group algebras. Even though in early stages in the works of Leptin and others,
the twisted group actions were defined in much more general setting (see, for example,
\cite{EL} and \cite{Lep}), in recent years most authors have considered these concepts for the case where $\Om=\Om_\T$,
i.e. when the 2-cocycle has valued in the unit circle.
However, since we are mostly interested to see when Orlicz spaces on a locally compact group $G$ are algebras and study their algebraic properties, the assumption $\Om=\Om_\T$
is not enough as the corresponding Orlicz space is rarely an algebra with respect to the twisted convolution coming from $\Om_\T$ (except in the trivial case when $G$ is compact, see for example \cite{HKM}). Hence we are considering a more general setting
which is in fact needed for our study and to make it work. We will demonstrate this in more details later in Section \ref{S:Twisted Orlicz alg}.

In the particular case where $\Om\in \Zbw$, we can have an alternative representation
of the twisted group algebra associated to $\Om$. More precisely, suppose that $\om : G \to \R_+$ is a weight associated to $|\Om|$ as in Definition
\ref{D:bound 2 cocycles} and consider the {\it weighted $L^1$-space}
\begin{align}
L^1_\om(G):=\{ f:G \to \C : f\om \in L^1(G)\}.
\end{align}
Then $L^1_\om(G)$ with the norm $\|f\|_\om=\|f\om\|_1$ is a Banach space.
Moreover, if $\tw$ and $\tc$ are twisted convolutions with respect to $\Om$ and $\Om_\T$, respectively, then it is straightforward
to verify that $(L^1_\om(G),\tc, \|\cdot\|_\om)$ becomes a Banach algebra so that the mapping
$$\Lambda_\om: (L^1(G),\tw, \|\cdot\|_1) \to (L^1_\om(G),\tc, \|\cdot\|_\om)$$
defined by
$$ \Lambda_\om(f)=f/\om \ \ \ (f\in L^1(G)).$$
is an isometric algebra isomorphism. Furthermore, it is easy to see that
the inclusion $$\iota: (L^1_\om(G),\tc, \|\cdot\|_\om) \hookrightarrow (L^1(G),\tc, \|\cdot\|_1)$$
is a continuous injective algebra homomorphism with the dense image.
We would use these relations 
as we see appropriate throughout this article.

We finish this section by pointing out that even though the assumption $\Om\in \Zbw$ might seem too restrictive,
in most interesting cases, this always hold. One of those cases (which include all cases we discuss in this article) is when $G$ is amenable
as we will see in the following lemma. This is well-known but we present a short proof for the sake of completeness.

\begin{lem}\label{L:2-coboundary amenable}
Suppose that $G$ is amenable and $\Om \in \Zb$. Then $|\Om|$ is the 2-coboundary determined by a weight $\om : G \to \R_+$
with $\om(e)=1$ and $1/\om\in L^\infty(G)$. In other words,
$$\Zb=\Zbw.$$
\end{lem}

\begin{proof}
It is well-known that for an amenable group $G$, $\Hc^2(G,\R)=\{0\}$, where here we regard $\R$ as the group of additive real numbers.
Since $(\R,+)\cong (\R_+,.)$ as locally compact abelian groups, we conclude that $\Hc^2(G,\R_+)=\{0\}$, i.e. $|\Om|$
is a 2-coboundary. Let $\om:G \to \R_+$ be a weight function determining $|\Om|$. Then it follows that
$\om(e)=1$ and, for all $s,t\in G$, $\om(st)\leq C\om(s)\om(t)$, where $C=\|\Om\|_\infty$. In particular,
$C\om$ is a positive submultiplicative measurable function on $G$, and so, by the amenability of $G$,
there is a continuous positive value character $\chi$ on $G$ such that $C\om \geq \chi$ a.e. (\cite{W}). Thus if we replace
$\om$ with $\om/\chi$, we get our desired result.
\end{proof}


\section{Twisted Orlicz algebras}\label{S:Twisted Orlicz alg}

Throughout the rest of the paper, we assume that $(\Phi,\Psi)$ is a pair of complementary Young function with $\Phi$ being continuous on $[0,\infty)$ and positive on $(0,\infty)$.

In this section, we would like to find sufficient conditions under which the twisted convolution \eqref{Eq:twisted convolution} turns an Orlicz space to an algebra as
formulated in the following definition.

\begin{defn}
Let $G$ be a locally compact group, let $\Om\in \Zb$, and let $\tw$ be the twisted convolution coming from $\Om$. We say
that $(L^\Phi(G),\tw)$ is a {\bf twisted Orlicz algebra} if $(L^\Phi(G),\tw, \|\cdot\|_\Phi)$ is a Banach algebra, i.e. there is $C>0$ such that
for every $f,g\in L^\Phi(G)$, $f\tw g\in L^\Phi(G)$ with
$$ \|f\tw g\|_\Phi \leq C\|f\|_\Phi \|g\|_\Phi.$$
\end{defn}

The following lemma which is a generalization of \cite[Proposition 1, p.384]{rao2}
shows that one always has a natural $L^1(G)$-bimodule structure on $(L^\Phi(G),\tw)$.


\begin{lem}\label{L:Orlicz L1 mod}
Let $G$ be a locally compact group, and let $\Om\in \Zb$. Then:\\
$(i)$ There is $C>0$ such that for all $f\in L^\Phi(G)$ and $s\in G$, both $\delta_s\tw f$ and $f\tw \delta_s$, defined in
\eqref{Eq:twisted convolution-left right translation}, belong to $L^\Phi(G)$ with
$$\|\delta_s\tw f\|_{\Phi} \leq C\|f\|_{\Phi} \ \text{and}\ \|f\tw \delta_s\|_{\Phi} \leq C\|f\|_{\Phi};$$
$(ii)$ $L^\Phi(G)$ is a Banach $L^1(G)$-bimodule with respect to the twisted convolution \eqref{Eq:twisted convolution};\\
$(iii)$ $\Sm^\Phi(G)$ becomes an essential Banach $L^1(G)$-submodule of $L^\Phi(G)$ with respect to the twisted convolution \eqref{Eq:twisted convolution}.
\end{lem}

\begin{proof}
(i) For every $f\in L^\Phi(G)$ and $t\in G$, we set 
$$L_s f(t)=f(s^{-1}t) \ \ \text{and}\ \ R_sf(t)=f(ts)\Delta(s),$$
where $\Delta$ is the modular function of $G$. It follows from the formulas
\eqref{Eq:Orlicz defn}, \eqref{Eq:Orlicz norm}, and the standard properties of the Haar measure that both $L_sf$ and $R_sf$ belong to $L^\Phi(G)$ with
$\|L_sf\|_{\Phi}=\|R_sf\|_{\Phi}=\|f\|_{\Phi}$. On the other hand, by our hypothesis,
$$|\delta_s\tw f|\leq C|L_sf| \ \ \text{and}\ \ |f\tw \delta_s|\leq C|R_{s^{-1}}f|,$$
where $C=\|\Om\|_\infty$. Therefore $\delta_s\tw f$ and $f \tw \delta_s$ belong to $L^\Phi(G)$ with
$$\|\delta_s\tw f\|_{\Phi}\leq C\|L_sf\|_{\Phi}=C\|f\|_{\Phi} \ \ \text{and}\ \ \|f\tw \delta_s\|_{\Phi}\leq C\|R_{s^{-1}}f\|_{\Phi}=C\|f\|_{\Phi}.$$
(ii) For every $f\in L^1(G)$, $g\in L^\Phi(G)$, and $h\in L^\Psi(G)$
with $\int_G \Psi(|h(s)|) ds \leq 1$, we have
\begin{eqnarray*}
\int_G |(f\tw g)(t)h(t)| dt &\leq & \int_G\int_G |f(s)g(s^{-1}t)\Om(s,s^{-1}t)h(t)| dsdt
\\ &=&
\int_G |f(s)|\int_G |(\delta_s\tw g)(t)h(t)|dtds \ \ (\text{by}\ \eqref{Eq:Orlicz norm})\\
&\leq & \int_G |f(s)| \|\delta_s\tw g\|_\Phi ds \ \ (\text{by part(i)}) \\
&\leq & C\int_G |f(s)| \|g\|_\Phi ds \\
&=& C\|f\|_1\|g\|_\Phi.
\end{eqnarray*}
Therefore, again by \eqref{Eq:Orlicz norm}, it follows that $\|f\tw g\|_\Phi\leq C \|f\|_1\|g\|_\Phi$ so that
$L^\Phi(G)$ is a Banach left $L^1(G)$-module. The other case (Banach right module) follows similarly considering that for every $t\in G$,
$$(g\tw f)(t):=\int_G f(s)g(ts^{-1})\Om(ts^{-1},s)\Delta(s^{-1})ds=\int_G f(s)(g\tw \delta_s)(t)ds.$$
(iii) Suppose that $f,g\in C_c(G)$ and $\alpha>0$. Since $\Phi$ is a positive continuous convex function on $\R^+$, it is increasing. Hence, for every $t\in G$,
$$\Phi(|\alpha (f\tw g)(t)|) \leq \Phi(\alpha \lambda(\supp\,f) \|f\|_\infty \|g\|_\infty \|\Om\|_\infty).$$
Therefore
\begin{eqnarray*}
\int_G \Phi(|\alpha (f\tw g)(t)|)dt &=& \int_K \Phi(|\alpha (f\tw g)(t)|)dt \\
&\leq & \lambda(K)\Phi(\alpha \lambda(\supp\,f) \|f\|_\infty \|g\|_\infty \|\Om\|_\infty),
\end{eqnarray*}
where $K$ is a compact set containing $\supp\,f\supp\,g$. Thus, by \cite[Corollary 3.4.4]{rao}, $f\tw g\in \Sm^\Phi(G)$.
The rest follows from part (ii) and the fact that $C_c(G)$ is norm dense in $(L^1(G),\|\cdot\|_1)$ and
($\Sm^\Phi(G),\|\cdot\|_\Phi)$.


\end{proof}

The following theorem provides a key step in our approach to obtain twisted Orlicz algebras. Roughly speaking, it states that one could get a twisted Orlicz algebra if
the 2-cocycle $|\Om|$ is dominated by the sum of two suitable positive functions.

\begin{thm}\label{T:twisted Orlicz alg}
Let $G$ be a locally compact group, and let $\Om\in \Zb$. Suppose that there exist non-negative measurable functions $u$ and $v$
in $L^\Psi(G)$ such that
\begin{align}\label{Eq:2-cocycle bdd sum}
|\Om(s,t)|\leq u(s)+v(t) \ \ \ (s,t\in G).
\end{align}
Then for every $f,g\in L^\Phi(G)$, the twisted convolution \eqref{Eq:twisted convolution}
is well-defined on $L^\Phi(G)$ and we have
\begin{align}\label{Eq:twisted convolution-diff norm relation}
\|f\tw g\|_\Phi\leq \|fu\|_1 \|g\|_\Phi+\|f\|_\Phi\|gv\|_1.
\end{align}
In particular, $(L^\Phi(G),\tw)$ becomes a twisted Orlicz algebra having $\Sm^\Phi(G)$ as a closed subalgebra.
\end{thm}

\begin{proof}
Fix $f,g\in L^\Phi(G)$. Then, for every $t\in G$,
\begin{eqnarray*}
\int_G |f(s)g(s^{-1}t)\Om(s,s^{-1}t)| ds &\leq &
\int_G |f(s)g(s^{-1}t)||u(s)| ds\\
&+& \int_G |f(s)g(s^{-1}t)||v(s^{-1}t)| ds\\
&=& |fu|*|g|(t)+|f|*|gv|(t).
\end{eqnarray*}
Since from the H\"{o}lder's inequality \eqref{Eq:Holder inequality} both $fu$ and $gv$ belong
to $L^1(G)$ and, by Lemma \ref{L:Orlicz L1 mod}, $L^\Phi(G)$ is a Banach $L^1(G)$-bimodule under the convolution, it follows that
the measurable function
$$t \mapsto (f\tw g) (t)=\int_G f(s)g(s^{-1}t)\Om(s,s^{-1}t)ds$$
belongs to $L^\Phi(G)$. Moreover, for every $h\in L^\Psi(G)$ with $\int_G \Psi(|h(s)|) ds \leq 1$, we have
\begin{eqnarray*}
\int_G |(f\tw g)(t)h(t)| dt &\leq & \int_G\int_G |f(s)g(s^{-1}t)\Om(s,s^{-1}t) h(t)| dsdt
\\ &\leq &
\int_G |fu|*|g|(t) |h(t)| dt + \int_G |f|*|gv|(t)||h(t)| dt \\
&\leq & \||fu|*|g|\|_\Phi+\||f|*|gv|\|_\Phi \ \ (\text{by Lemma}\ \ref{L:Orlicz L1 mod})
\\ &\leq& \|fu\|_1 \|g\|_\Phi+\|f\|_\Phi\|gv\|_1.
\end{eqnarray*}
Therefore, by \eqref{Eq:Orlicz norm}, we obtain the relation \eqref{Eq:twisted convolution-diff norm relation}.
This, together with the repeated use of H\"{o}lder's inequality \eqref{Eq:Holder inequality}, implies that
$$\|f\tw g\|_\Phi \leq C \|f\|_\Phi\|g\|_\Phi,$$
where $C=N(u)_\Psi+N(v)_\Psi$. Hence  $(L^\Phi(G),\tw)$ becomes a twisted Orlicz algebra.
\end{proof}

We have the following immediate corollary. We recall that a function
$\fL : G \to \R^+$ is called {\it weakly subadditive} if there is $C>0$ such that
\begin{eqnarray}\label{Eq:weak subadditive relation}
\fL(st)\leq C(\fL(s)+\fL(t)) \ \ (s,t\in G).
\end{eqnarray}

\begin{cor}\label{C:twisted Orlicz alg-weak add weight}
Let $G$ be a locally compact group, let $\Om\in \Zb$, and let $\tw$ be the twisted convolution
coming from $\Om$. Suppose that
\begin{align}\label{Eq:2-cocycle bdd by subadditive function}
|\Om(s,t)|\leq \frac{\fL(st)}{\fL(s)\fL(t)} \ \ \ (s,t\in G),
\end{align}
where $\fL: G \to \R^+$ is a weakly subadditive function with $1/\fL \in L^\Psi(G)$.
Then $(L^\Phi(G),\tw)$ is a twisted Orlicz algebra.
\end{cor}

\begin{proof}
Since $\fL$ is weakly subadditive, it satisfies \eqref{Eq:weak subadditive relation}. Therefore, if we combine this with \eqref{Eq:2-cocycle bdd by subadditive function}, we get
$$|\Om(s,t)|\leq \frac{C}{\fL(s)}+\frac{C}{\fL(t)} \ \ \ (s,t\in G).$$
Hence if we put $$u=v=C/\fL,$$
then $\Om$ satisfies \eqref{Eq:2-cocycle bdd sum} with $u,v\in L^\Psi(G)$. It now follows from Theorem \ref{T:twisted Orlicz alg} that $(L^\Phi(G),\tw)$ is a twisted Orlicz algebra.
\end{proof}

The preceding corollary gives us a very useful tool to determine when we have twisted Orlicz algebras. We will apply it mostly on compactly generated groups of polynomial growth (see Section \ref{S:Groups poly. growth}). However, as it is demonstrated in the following example, it can be applied to other classes of groups as well. The example is taken from \cite[Example 1 and Remark 2]{P}.

\begin{exm}\label{E:twist-locally finite}
Let $G$ be a locally compact group for which there is an increasing sequence  $\{G_i\}_{i\in \N}$ of compact subgroups of $G$ such that $G:=\cup_{i\in \N} G_i$.
Take an increasing sequence $\{n_i\}_{i\in \N}\in [1,\infty)$. Define
$\om:G \to [1,\infty)$ by
$$\om=1+\sum_{i=1} n_i 1_{G_{i+1}\setminus G_{i}}.$$
It is easy to see that
$$\om(st)=\max\{ \om(s),\om(t) \} \ \ \ (s,t\in G).$$
This, in particular, implies that $\om$ is a weakly additive weight on $G$. Moreover, we can pick $\{n_i\}$ in such a way that $1/\om \in L^1(G)\cap L^\infty(G)\subseteq L^\Psi(G)$, where the inclusion is an easy consequence of \eqref{Eq:Orlicz defn}. Therefore, in this case, by Corollary \ref{C:twisted Orlicz alg-weak add weight},
$(L^\Phi(G),\tw)$ is twisted Orlicz algebra, where $\Om\in \Zb$ such that
$|\Om|$ is determined by $\om$.
\end{exm}

\section{symmetry}\label{S:symmetry}

In this section, we investigate the symmetry for twisted Orlicz algebras.
The notion of symmetry plays an important role in the theory of Banach $*$-algebras.
Let A be a Banach $*$-algebra. We say that $A$ is {\it symmetric} if for every $a\in A$,
$\sg_A(a^*a)\subseteq [0,\infty)$, where $\sg_A(b)$ is the spectrum of element $b\in A$.
It is well-known that C$^*$-algebras are symmetric and a commutative Banach $*$-algebras
is symmetric if and only if every multiplicative linear functional on  $A$ is a
$*$-homomorphism.
Also the group algebra of a compactly generated group with polynomial growth is symmetric
\cite{Los}, whereas the group algebra of the free group on $n$ generators is not symmetric for $n\geq 2$.

In order for us to look at the symmetry for twisted Orlicz algebras, we first need to restrict
ourselves to those weights for which a natural involution can be defined on the twisted group
algebra.

\begin{defn}\label{D:bound symm 2 cocycles}
Let $G$ be a locally compact group. We denote $\Zbs$ to be the group of {\bf bounded symmetric 2-cocycles on $G$ with values in $\Cm$}
which consists of all elements $\Om\in \Zbw$ for which there is a weight $\om$ associated to $|\Om|$ such that
$$\om(s)=\om(s^{-1}) \ \ \ (s\in G).$$
Such weights on $G$ are called {\bf symmetric}.
\end{defn}

Now suppose that $\Om \in \Zbs$, $\om$ is a symmetric weight associated to $|\Om|$, and
$\tw$ and $\tc$ are twisted convolutions with respect to $\Om$ and $\Om_\T$, respectively.
It is well-known and easily seen that $(L^1(G),\tc,\|\cdot\|_1)$ is a Banach $*$-algebra with the involution defined by
\begin{eqnarray}\label{Eq:Involution}
f^*(s)=\overline{f(s^{-1})}\Delta(s^{-1})\overline{\Om_\T(s,s^{-1})} \ \ \ (f\in L^1(G), s\in G).
\end{eqnarray}
On the other hand, by what was discussed in Section \ref{S:Twisted group alg}, the two Banach algebras
$(L^1(G),\tw, \|\cdot\|_1)$ and $(L^1_\om(G),\tc, \|\cdot\|_\om)$ are the same and the latter can be viewed
as a subalgebra of $(L^1(G),\tc, \|\cdot\|_1)$. This, in particular, allows us to define the involution
on $L^1_\om(G)$ as the restricted one from $L^1(G)$ and it is routine to verify that $(L^1_\om(G),\tc, \|\cdot\|_\om)$
becomes a Banach $*$-algebra. Hence, when it comes to symmetry, it is more useful to consider the ``weighted" representation
for twisted group algebras.

Our next step is to generalize the preceding discussion to twisted Orlicz spaces so that they can have
an involutive structure. We present it in the following lemma whose proof is straightforward so we omit it.

\begin{lem}\label{L:twist covn-weight conv}
Let $G$ be a locally compact group, let $\Om\in \Zbw$, and let $\om$ be a weight associated to $|\Om|$. Define
the {\it weighted $L^\Phi$-space}
\begin{align}
L^\Phi_\om(G):=\{ f:G \to \C : f\om \in L^\Phi(G)\}.
\end{align}
Then $L^\Phi_\om(G)$ with the norm $\|f\|_{\Phi,\om}=\|f\om\|_\Phi$ is a Banach space.
Moreover, if $\tw$ and $\tc$ are twisted convolutions with respect to $\Om$ and $\Om_\T$, respectively, then
 $(L^\Phi_\om(G),\tc, \|\cdot\|_{\Phi,\om})$ becomes a Banach $(L^1_\om(G),\tc, \|\cdot\|_\om)$-bimodule so that the mapping
\begin{align}\label{Eq:twist covn-weight conv-mapping}
\Lambda_\om: L^\Phi(G) \to L^\Phi_\om(G) \ , \ \Lambda_\om(f)=f/\om
\end{align}
is a linear isometric isomorphism satisfying $(f\in L^1(G), g\in L^\Phi(G))$
\begin{align}\label{Eq:twist covn-weight conv-relation}
\Lambda_\om(f\tw g)=\Lambda_\om(f)\tc \Lambda_\om(g) \ \ \ \text{and} \ \ \Lambda_\om(g\tw f)=\Lambda_\om(g)\tc \Lambda_\om(f).
\end{align}
If, in addition, $(L^\Phi(G),\tw,\|\cdot\|_\Phi)$ is a twisted Orlicz algebra, then \eqref{Eq:twist covn-weight conv-relation} holds for every
$f,g\in L^\Phi(G)$ so that  $(L^\Phi_\om(G),\tc, \|\cdot\|_{\Phi,\om})$ is a Banach algebra.
Furthermore, if $\om$ is symmetric and $G$ is unimodular, then, with the involution \eqref{Eq:Involution},
$(L^\Phi_\om(G),\tc)$ becomes a Banach $*$-algebra with either norms \eqref{Eq:Orlicz norm} or \eqref{Eq:Orlicz Luxemburg defn}.
\end{lem}

We remark that in the particular case where $\Phi(x)=x^p/p$ ($1<p<\infty$) and $\Om_\T=1$, then the space $L^\Phi_\om(G)$
in Lemma \ref{L:twist covn-weight conv} is precisely $L^p_\om(G)$ with the convolution considered in \cite{KM}.

We are now ready to investigate the symmetry of twisted Orlicz algebras. As we have seen so far, our approach is to first
determine this property for twisted group algebras. This is the generalization of the approach in \cite{KM}
applying to symmetry of $L^p_\om(G)$ with the (untwisted) convolution.

\subsection{Symmetry of twisted group algebras}\label{S:Symmetry of twisted group algebras}

In order to investigate the symmetry for twisted group algebras, we rely on a natural relation that exist
between the twisted convolution on a locally compact group $G$ and the standard (untwisted) convolution on a certain central extension of $G$ \cite[Section 3]{EL}. We can then apply what is known for symmetry of weighted group algebras to obtain
our results. This is similar to the approach taken as a special case in \cite[Section 2.1]{GL}. Below, we present it in a more general setting:

Let $G$ be a locally compact group, let $\Omega\in \Zbw$, and let $\GT:=G\times \T$. The $\GT$ becomes a group with the action
$$(s,\alpha)\cdot (t,\beta):=(st, \alpha\beta \Omega_\T(s,t)) \ \ (s,t\in G, \alpha, \beta \in \T).$$
Moreover, there is a locally compact topology on $\GT$ so that it becomes a locally compact group
with the topology coinciding with the usual product topology (The separable case is due to G. W. Mackey \cite{M}
and nonseparable one follows from the work of M. Leinert \cite{L}). In particular, the Haar measure on $\GT$ is just the
product Haar measure $\int_G\int_\T \cdots dxd\alpha.$ Now consider the mapping
$$\Gamma: C_c(G) \to L^1(\GT) \ \ , \ \ \Gamma(f)(s,\alpha)=\overline{\alpha}f(s).$$
If $\tc$ is the twisted convolution on $L^1(G)$ coming from $\Omega_\T$, and $*$ is the
(ordinary) convolution on $\GT$, then it is straightforward to check the following (see, \cite[Lemma 3.2]{EL} or
\cite[Lemma 2.3 and Corollary 2.5]{GL}):

(1) $\Gamma$ extends to an isometric $*$-algebra isomorphism from $(L^1(G),\tc)$ into $(L^1(\GT),*)$;

(2) $Im\, \Gamma$ is an ideal in $(L^1(\GT),*)$.\\
Moreover, if $\om$ is a symmetric weight associated to $|\Om|$ and we consider  $\tilde{\om}:\GT \to \R^+$ defined by
\begin{align}\label{Eq:central extention weight}
\tilde{\om}(s,\alpha)=\om(s) \ \ (s\in G, \alpha \in \C),
\end{align}
then it is straightforward to verify $\tilde{\om}$ is a symmetric weight on $\GT$ so that the preceding statements (1) and (2)
remain valid if we replace  $(L^1(G),\tc)$ and $(L^1(\GT),*)$ with their weighted analog $(L^1_\om(G),\tc)$ into $(L^1_{\tilde{\om}}(\GT),*)$,
respectively. This observation allows us to recover
symmetry of the twisted group algebra from the one's of the weighted group algebra. We can summarize all these discussions in the following.

\begin{prop}\label{P:symm-twisted group alg}
Let $G$ be a locally compact group, let $\Om\in \Zbs$, and let $\om$ be a symmetric weight associated to $|\Om|$.
The twisted group algebra $(L^1_\om(G),\tc)$ can be viewed as a $*$-closed two-sided ideal with a bounded approximate identity of the (untwisted) weighted group algebra  $(L^1_{\tilde{\om}}(\GT),*)$.
In particular, if $(L^1_{\tilde{\om}}(\GT),*)$ is symmetric, then so is $(L^1_\om(G),\tc)$.
\end{prop}

\subsection{Symmetry of twisted Orlicz algebras}

A very useful technique to investigate the symmetry of a Banach algebra as well as being inverse-closed
in the concept of differential norm. Assume that $A\subseteq B$ are two Banach algebras with a common unit element. A {\it differential norm} is a norm on A that satisfies
\begin{align}\label{Eq:differential norm relation}
\|ab\|_A\leq C(\|a\|_A\|b\|_B+\|a\|_B\|b\|_A)
\end{align}
for all $a, b \in A$. In this case, we call $A$ {\it a differential subalgebra of} $B$. This concept has appeared
in various articles however, the preceding formulation is given in \cite[Section 3.1]{GK}. The following lemma demonstrates
the main property of differential subalgebras which we need for our purpose. The proof is well-known, has been
presented in several articles and somewhat straightforward so we omit it
(see, for example, \cite[Lemma 3.2]{GK} and the references referred in).

\begin{lem}\label{L:symm-diff subalgebras}
Let $A$ be a differential subalgebra of a Banach algebra $B$, and let $r_A$ and $r_B$
be the spectral radius function of $A$ and $B$, respectively.
Suppose that either $A$ and $B$ are simultaneously unital with the same unit or
they are both non-unital.
Then:\\
$(i)$ For every element $a\in A$, $r_A(a)=r_B(a)$;\\
$(ii)$ $A$ is inverse-closed in $B$;\\
$(iii)$ Suppose that $B$ is a Banach $*$-algebra having $A$ as a $*$-subalgebra. If $B$ is symmetric, then so is $A$.
\end{lem}

The following theorem which is the main result of this section demonstrates that certain twisted Orlicz algebras can be viewed as differential subalgebras of twisted group algebras. Thus we could determine their symmetry by applying the preceding Lemma together with the classical results concerning the symmetry of twisted group algebras discussed in Section \ref{S:Symmetry of twisted group algebras}.

\begin{thm}\label{T:twisted Orlicz alg-weak add-symm}
Let $G$ be a locally compact unimodular group, let $\Om\in \Zbs$, and let $\sg$ be a symmetric weight associated to $|\Om|$. Suppose that there exists a symmetric weakly subadditive weight $\om$ on $G$
with $1/\om \in L^\Psi(G)$ and $M>0$ such that
\begin{align}\label{Eq:2-cocycle bdd by subadditive weight}
\frac{\sg(st)}{\sg(s)\sg(t)}\leq \frac{M\om(st)}{\om(s)\om(t)} \ \ \ (s,t\in G).
\end{align}
Then $\rho:=\sg/\om$ defines a symmetric weight on $G$ so that $(L^\Phi_\sg(G),\tc, \|\cdot\|_{\Phi,\sg})$ becomes differential $*$-subalgebra of $(L^1_\rho(G),\tc, \|\cdot\|_{1,\rho})$,
where $\tc$ is the twisted convolution coming from $\Om_\T$.
In particular, $(L^\Phi_\sg(G),\tc, \|\cdot\|_{\Phi,\sg})$ is symmetric whenever $(L^1_\rho(G),\tc, \|\cdot\|_{1,\rho})$ is symmetric.
\end{thm}


\begin{proof}
We first note that the relation \eqref{Eq:2-cocycle bdd by subadditive weight} is noting but \eqref{Eq:2-cocycle bdd by subadditive function} by putting $\fL=\om/M$. Thus, by our hypothesis and Corollary \ref{C:twisted Orlicz alg-weak add weight}, $(L^\Phi(G),\tw, \|\cdot\|_{\Phi})$ is a twisted Orlicz algebra, where $\tw$ is the twisted convolution coming from $\Om$. Therefore, by using the equivalent formulation given by Lemma \ref{L:twist covn-weight conv}, $(L^\Phi_\sg(G),\tc, \|\cdot\|_{\Phi,\sg})$ is a Banach $*$-algebra. Now consider the function $\rho:G\to \R^+$ given by
\begin{align}\label{Eq:weight mod weak subadd weight}
\rho(s)=\frac{\sg(s)}{\om(s)} \ \ (s\in G).
\end{align}
We first show that $\rho$ is a symmetric weight on $G$. It is clear that $\rho$ is measurable and symmetric
with $\rho(e)=1$. Also, since $\om$ is a symmetric weight, it is bounded away from 0, i.e. there is $K>0$ so that $$\om(s)\geq K \ \ (s\in G).$$
Therefore, by \eqref{Eq:weight mod weak subadd weight}, $\rho \leq \sg/K$, and so, $\rho$ is locally integrable
since $\sg$ is locally integrable. Finally,
the relation \eqref{Eq:2-cocycle bdd by subadditive weight} is clearly equivalent to
$$\rho(st)\leq M\rho(s)\rho(t) \ \ \  (s,t\in G).$$
That is $\rho$ is a weight on $G$. Our next step is to show that $(L^\Phi_\sg(G),\tc, \|\cdot\|_{\Phi,\sg})$ is differential $*$-subalgebra of $(L^1_\rho(G),\tc, \|\cdot\|_{1,\rho})$. We first note that
\begin{align}\label{Eq:weighted Orlicz space subset weighted L1}
L^\Phi_\sg(G)\subseteq L^1_\rho(G).
\end{align}
To see this, suppose that $f\in L^\Phi_\sg(G)$. Since, by hypothesis,
$1/\om \in L^\Psi(G)$, we have
\begin{eqnarray*}
\|f\|_{1,\rho}&=& \int_G |f(s)|\rho(s) ds \\
&=& \int_G |f(s)|\sg(s)\frac{1}{\om(s)} ds \\
 &\leq & \|f\sg\|_\Phi N_\Psi(1/\om) \ \ (\text{by H\"{o}lder's inequality} \ \eqref{Eq:Holder inequality})\\
 &=& \|f\|_{\Phi,\sg} N_\Psi(1/\om).
\end{eqnarray*}
Hence \eqref{Eq:weighted Orlicz space subset weighted L1} holds. This, in particular, implies that
$(L^\Phi_\sg(G),\tc)$ is a $*$-subalgebra of $(L^1_\rho(G),\tc)$. It now remains to
show that the corresponding relation \eqref{Eq:differential norm relation} holds. Since $\om$ is weakly subadditive, it satisfies \eqref{Eq:weak subadditive relation} for a fixed constant  $C>0$. Thus, if we apply \eqref{Eq:2-cocycle bdd by subadditive weight}, for every $s,t\in G$, we have
$$|\Om(s,t)|=\frac{\sg(st)}{\sg(s)\sg(t)}\leq \frac{M\om(st)}{\om(s)\om(t)}\leq \frac{CM}{\om(s)}+\frac{CM}{\om(t)}=u(s)+u(t),$$
where $u:=CM/\om$. Since, by our hypothesis, $u\in L^\Psi(G)$, it follows from \eqref{Eq:twisted convolution-diff norm relation} that
$$\|f\tw g\|_{\Phi} \leq \|fu\|_{1} \|g\|_{\Phi}+\|f\|_{\Phi}\|gu\|_{1} \ \ \ (f,g\in L^\Phi(G)).$$
Alternatively, if we apply Lemma \ref{L:twist covn-weight conv} and use the equivalent weighted reformulation, we get
\begin{eqnarray*}
\|f\tc g\|_{\Phi,\sg} &\leq& \|f\sg u\|_{1} \|g\sg\|_{\Phi}+\|f\sg\|_{\Phi}\|g\sg u\|_{1} \\
&=& CM\Big(\|f\|_{1,\rho} \|g\|_{\Phi,\sg }+\|f\|_{\Phi,\sg} \|g\|_{1,\rho}\Big),
\end{eqnarray*}
for every $f,g\in L^\Phi_\sg(G)$. Hence  $(L^\Phi_\sg(G),\tc, \|\cdot\|_{\Phi,\sg})$ is a differential $*$-subalgebra of $(L^1_\rho(G),\tc, \|\cdot\|_{1,\rho})$. The final statement follows from
Lemma \ref{L:symm-diff subalgebras}.
\end{proof}

We finish this section by pointing out that the preceding theorem is particularly useful when investigating symmetry of twisted Orlicz algebra of compactly generated groups with polynomial growth as demonstrated in the following section. However, we can also apply it to other cases.

\begin{exm}
Let $G$ and $\om$ be the ones considered in Example \ref{E:twist-locally finite}
with the assumption that $1/\om\in L^1(G)\cap L^\infty(G)\subseteq L^\Psi(G)$.
Suppose that $\GT$ is the central extension of $G$ considered in Section \ref{S:Symmetry of twisted group algebras}. Take $\Om_\T\in \ZTb$, $p>0$ and put $\rho=\om^p$. It is
clear that $\tilde{\rho}$, defined by \eqref{Eq:central extention weight}, is a weakly additive symmetric weight on $\GT$ with $\tilde{\rho}^{-p} \in L^1(\GT)$. Therefore, by \cite[Theorem 1]{P}, $(L_{\tilde{\rho}}^1(\GT),*)$ is symmetric, and so, by Theorem \ref{T:symm-twisted group alg}, $(L_{{\rho}}^1(G),\tc)$ is symmetric, where $\tc$ is the twisted convolution coming from $\Om_\T$. Hence if we apply Theorem \ref{T:twisted Orlicz alg-weak add-symm} for $\sg:=\om^{p+1}$ and $\rho:=\om^p$, it follows that $(L_{\om^{p+1}}^\Phi(G),\tc)$ is a symmetric twisted Orlicz algebra.
\end{exm}

\section{Groups with polynomial growth}\label{S:Groups poly. growth}

\subsection{General theory}\label{S:Groups poly. growth-definition}
Let $G$ be a compactly generated group with a fixed compact symmetric generating neighborhood $U$ of the identity of the group $G$.
$G$ is said to have {\it polynomial growth} if there exist $C>0$ and $d\in \N$ such that for every $n\in \N$
	$$\lambda(U^n)\leq Cn^d \ \ \ (n\in \N).$$
Here $\lambda(S)$ is the Haar measure of any measurable $S\subseteq G$ and
	$$U^n=\{u_1\cdots u_n : u_i\in U, i=1,\ldots, n \}.$$
The smallest such $d$ is called {\bf the order of growth} of $G$ and it is denoted by $d(G)$.
It can be shown that the order of growth of $G$ does not depend on the symmetric generating set $U$, i.e. it is a universal constant for $G$.
Also, by \cite[Lemma 2.3]{FGL}, the compact symmetric neighborhood $U$ can be chosen so that it has a strict polynomial growth, i.e.
there are positive numbers $C_1$ and $C_2$ such that
\begin{align}\label{Eq:stric poly growth}
C_1n^d \leq \lambda(U^n)\leq C_2n^d \ \ \ (n\in \N).
\end{align}
It is immediate that compact groups are of polynomial growth. More generally, every $G$ with the property that the conjugacy class of
every element in $G$ is relatively compact has polynomial growth \cite[Theorem 12.5.17]{Pal}. Also every (compactly generated) nilpotent
group (hence an abelian group) has polynomial growth \cite[Theorem 12.5.17]{Pal}.

Using the generating set $U$ of $G$ we can define a {\it length function} $\tau_U : G \to [0, \infty)$ by
\begin{align}\label{Eq:length function}
\tau_U(x)=\inf \{n\in \N : x\in U^n \} \ \ \text{for} \ \ x \neq e, \ \ \tau_F(e)=0.
\end{align}
When there is no fear of ambiguity, we write $\tau$ instead of $\tau_U$.
It is straightforward to verify that $\tau$ is a symmetric subadditive function on $G$, i.e.
\begin{align}\label{Eq:lenght func-trai equality}
\tau(xy)\leq \tau(x)+\tau(y) \ \ \text{and} \ \ \tau(x)=\tau(x^{-1})\ \ (x,y\in G).
\end{align}
We can use $\tau$ to define various weights on $G$.
More precisely, for every $0< \alpha \leq 1$, $\beta \geq 0$, $\gamma >0$, and $C>0$,
we can define the {\it polynomial weight} $\om_\beta$ on $G$ of order $\beta$ by
\begin{align}\label{Eq:poly weight-defn}
\om_\beta(s)=(1+\tau(s))^\beta  \ \ \ \ (s\in G),
\end{align}
and the {\it subexponential weights} $\sg_{\alpha, C}$ and $\rho_{\beta,C}$ on $G$ by
\begin{align}\label{Eq:Expo weight-defn}
\sg_{\alpha,C}(s)=e^{C\tau(x)^\alpha} \ \ \ \ (s\in G)
\end{align}
and
\begin{align}\label{Eq:Expo weight II-defn}
\rho_{\gamma,C}(s)=e^\frac{C\tau(s)}{(\ln (1+\tau(s)))^\gamma} \ \ \ \ (s\in G).
\end{align}

\subsection{Symmetric twisted Orlicz algebras over groups with polynomial growth}\label{S:Symm-twisted Orlicz alg-PG groups}

Throughout the rest of this section, we assume that $G$ is a compactly generated group of polynomial growth.

Let $\om$ be a symmetric weight on $G$.
We say that $\om$ satisfies the {\bf GRS-condition} if for every $s\in G$,
\begin{align*}\label{Eq:weight-GRS condition}
\lim_{n\to \infty} \om(s^n)^{1/n}=1.
\end{align*}
In \cite{FGL}, it is shown that the GRS-condition (among several others) characterizes precisely the symmetry of $L^1_\om(G)$ (see also
\cite{FGLLM}). Their work heavily depends on a method develops by Hulanicki (\cite{Hul}) as well as the symmetry of $L^1(G)$
(\cite[Section 3]{FGL}). This result was applied in \cite{KM} to determine the symmetry of $L^p_\om(G)$ for several cases
including polynomial weights defined in \eqref{Eq:poly weight-defn}.

Our goal is to modify and extend the results in \cite{KM} to twisted Orlicz algebra over $G$. We start with extending the main
result of \cite{FGL} to twisted group algebras over $G$ by applying Proposition \ref{P:symm-twisted group alg}.

\begin{thm}\label{T:symm-twisted group alg}
Let $\Om\in \Zbw$, and let $\om$ be a symmetric weight associated to $|\Om|$.
If $\om$ satisfies the GRS-condition, then $(L^1_\om(G),\tc)$ is symmetric, where $\tc$ is the twisted convolution coming from $\Om_\T$.
\end{thm}

\begin{proof}
Let $\GT:=G\times \T$ with the locally compact group structure coming from $\Om_\T$ as explained in Section \ref{S:Symmetry of twisted group algebras}.
Since $G$ has polynomial growth and $\GT$ is just a compact extension of $G$, then so does $\GT$.
Also it is easy to see that $\tilde{\om}:=\om\times 1$ is a symmetric weight on $\GT$ satisfying the GRS-condition.
Hence, by \cite[Theorem 1.3]{FGL}, $(L^1_{\tilde{\om}}(\GT),*)$ is symmetric (with the ordinary convolution).
Thus, by Proposition \ref{P:symm-twisted group alg}, $(L^1_\om(G),\tc)$ is symmetric.
\end{proof}

Due to the more complicated nature of $L^\Phi_\om(G)$, it is not clear how one can have a result as good as Theorem \ref{T:symm-twisted group alg} for twisted Orlicz algebras. However, we can still
show that for lots of important classes of weights, including the polynomial and subexponential weights
defined in Section \ref{S:Groups poly. growth-definition}, we can determine when we have symmetric twisted Orlicz algebras.

We start with following theorem which deals with the case of weakly additive weights.

\begin{thm}\label{T:twisted Orlicz alg-weak additive weight}
Let $\Om_\T\in \ZTb$, let $\tc$ be the twisted convolution coming from $\Om_\T$, and let $\om$ be symmetric weakly subadditive weight
on $G$ such that  $1/\om \in L^\Psi(G)$. Then $(L_{\om}^\Phi(G), \tc)$ is a symmetric twisted Orlicz algebra.
\end{thm}

\begin{proof}
By Theorem \ref{T:symm-twisted group alg}, $(L^1(G),\tc)$ is symmetric. Hence, the result is an immediate consequence
of Theorem \ref{T:twisted Orlicz alg-weak add-symm} by putting $\sg=\om$.
\end{proof}

\begin{cor}\label{C:twisted Orlicz alg-poly weight}
Let $\Om_\T\in \ZTb$, let $\tc$ be the twisted convolution coming from $\Om_\T$, and let $\om_\beta$ be the polynomial weight defined
in \eqref{Eq:poly weight-defn}. Then $(L_{\om_\beta}^\Phi(G), \tc)$ is a symmetric twisted Orlicz algebra if $\beta> \frac{d}{l}$.
Here $d:=d(G)$ is the degree of the growth of $G$ and
$l\geq 1$ is such that $\lim_{x\to 0^+}\frac{\Psi(x)}{x^l}$ exists.
\end{cor}

\begin{proof}
Suppose that $\beta> \frac{d}{l}$. It is easy to check that $\om_\beta$ is weakly subadditive. In fact,
$$\om_\beta(st) \leq 2^\beta \Big(\om_\beta(s)+\om_\beta(t)\Big) \ \ (s,t\in G).$$
Hence, by Theorem \ref{T:twisted Orlicz alg-weak additive weight}, it suffices to show that $1/\om_\beta \in L^\Psi(G)$. Since $\lim_{x\to 0^+}\frac{\Psi(x)}{x^l}$ exists, $\frac{\Psi(x)}{x^l}$ is bounded
when $x$ approaches $0^+$. Thus there exist
$M,N>0$ such that
$$\Psi(x) \leq Mx^l \ \ \ \ (0\leq x \leq N).$$
In particular,
$$\Psi\Big(\frac{N}{\om_\beta(s)}\Big) \leq \frac{MN^l}{\om_\beta(s)^l} \ \ \ \ (s\in G).$$
Let $U$ be the symmetric neighborhood of the identity in $G$ for which
\eqref{Eq:stric poly growth} holds. In particular,
\begin{eqnarray}\label{Eq:stric poly growth-II}
\lambda(U^n)\sim n^d \ \ (n\in \N).
\end{eqnarray}
For $C:=MN^l$, we have
\begin{eqnarray*}
\int_{G} \Psi(\frac{N}{\om_\beta(s)}) ds &\leq & \int_{G} \frac{C}{(1+\tau(s))^{l\beta}} ds\\
&\leq & C+C\sum_{n=1}^\infty \int_{U^n\setminus U^{n-1}} \frac{1}{(1+\tau(s))^{l\beta}} ds \\
&= &  C+C\sum_{n=1}^\infty \displaystyle \frac{\lambda(U^n \setminus U^{n-1})}{(1+n)^{l\beta}}\\
&= &  C+C\sum_{n=1}^\infty \displaystyle \frac{\lambda(U^n)-\lambda(U^{n-1})}{(1+n)^{l\beta}}+\lim_{n\to \infty} \frac{\lambda(U^n)}{(1+n)^{l\beta}}\\
&\leq &  C+\frac{C\lambda(\{e\})}{2^{l\beta}}+C\sum_{n=1}^\infty \displaystyle \lambda(U^n)\Big[\frac{1}{(1+n)^{l\beta}}-\frac{1}{(2+n)^{l\beta}}\Big],
\end{eqnarray*}
as, by \eqref{Eq:stric poly growth-II}, $\lim_{n\to \infty} \displaystyle \frac{\lambda(U^n)}{(1+n)^{l\beta}}=0$. Furthermore, again by applying \eqref{Eq:stric poly growth-II}, the series in the last line of the preceding expression converges since
$$\lambda(U^n)\Big[\frac{1}{(1+n)^{l\beta}}-\frac{1}{(2+n)^{l\beta}}\Big] \sim \frac{n^d}{(1+n)^{l\beta+1}}\sim n^{d-l\beta-1}$$
and $d-l\beta<0$. This completes the proof (see \eqref{Eq:Orlicz defn}).
\end{proof}

\begin{rem}
(i) If $\Psi$ is a Young function, it is convex and positive. In this case, it is easy to see that the function $\frac{\Psi(x)}{x}$ is positive and decreasing on $\R^+$, and so, $\lim_{x\to 0^+}\frac{\Psi(x)}{x}$ exists. Therefore, in Corollary \ref{C:twisted Orlicz alg-poly weight}, we can always assume that $l=1$. Of course, we would like to pick the largest possible $l$ to get the most optimal estimation.

(ii) We recall from \cite[Page 20]{rao} that two Young functions $\Psi_1$ and $\Psi_2$ are
{\it strongly equivalent} and write $\Psi_1 \approx \Psi_2$ if there exists $0<a\leq b<\infty$ such that $$\Psi_1(ax)\leq \Psi_2(x)\leq \Psi_1(bx) \ \ \ (x\geq 0).$$
It is clear from the definition of the Orlicz space \eqref{Eq:Orlicz defn} that the
strongly equivalent Young functions generate the same Orlicz space. In particular,
this allows us to consider different strongly equivalent Young functions in our
computation, say for example, to verify the conditions in Corollary \ref{C:twisted Orlicz alg-poly weight}.
\end{rem}

\begin{exm}
Let $\Z^d$ be the group of $d$-dimensional integers.
A usual choice of generating set for $\Z^d$ is
	$$F=\{(x_1,\ldots, x_d) \mid x_i\in \{-1,0,1\} \},$$
It is straightforward to see that
	$$|F^n|=(2n+1)^d \ \ \ (n=0,1,2,\ldots)$$
so that the order of the growth of $\Z^d$ is $d$.

Now suppose that $\Om_\T: \Z^d \times \Z^d \to \T$ is a 2-cocycle, $\tc$ is the twisted convolution coming from $\Om_\T$
and $\om_\beta$ is the polynomial weight on $\Z^d$ defined in \eqref{Eq:poly weight-defn}.

(i) If $\Phi(x)=\frac{x^p}{p}$ ($1<p<\infty$), then $\Psi(x)=\frac{x^q}{q}$ with $\frac{1}{p}+\frac{1}{q}=1$ so that
$(L_{\om_\beta}^\Phi(\Z^d),\tc)$ is a symmetric Banach $*$-algebra provided that $\beta >d/q$.

(ii) Suppose that there is $\delta>0$ such that $\Psi''$ exists and is continuous on $[0,\delta]$ with
$\Psi_+'(0)=0$. Then, by applying repeatedly the l'Hospital's rule,
$$\lim_{x\to 0^+}\frac{\Psi(x)}{x^2}=\lim_{x\to 0^+}\frac{\Psi'(x)}{2x}=\frac{\Psi''(0)}{2}.$$
Therefore $(L_{\om_\beta}^\Phi(\Z^d),\tw)$ is a symmetric Banach $*$-algebra if $\beta >d/2$. These can be applied to various Young functions such as a few we point out below
(see \cite[Proposition 2.11]{ML} and \cite[Page 15]{rao}).

(1) If $\Phi(x)=x\ln (1+x)$, then $\Psi(x) \approx \cosh x-1$.

(2) If $\Phi(x)=\cosh x-1$, then $\Psi(x)\approx x\ln (1+x)$.

(3) If $\Phi(x)=e^x-x-1$, then $\Psi(x)=(1+x)\ln(1+x)-x$.

(4) If $\Phi(x)=(1+x)\ln(1+x)-x$, then $\Psi(x)=e^x-x-1$.

\end{exm}

\subsection{The case of subexponential weights}\label{S:subexp weights-symm Orlicz space}

Throughout this article, in order to obtain the (symmetric) twisted Orlicz algebra $(L^\Phi(G),\tw)$, we relay heavily on the existence of two suitable functions in $L^\Psi_+(G)$ whose sum dominates the 2-cocycle
$|\Om|$. However, it is not
immediately clear whether this can be done if the weight associated to $|\Om|$ is
not weakly additive. A similar obstacle occurred in \cite{LSS} while investigating
the possibility of the isomorphism of a weighted group algebra of a finitely generated group of polynomial growth with an operator algebra. There, it was needed for $|\Om|$ to be dominated by the sum of two absolutely square summable functions.

In \cite{LSS}, a method was developed to bypass the obstacle pointed out in the preceding paragraph. More precisely, it is shown in \cite[Theorem 3.3]{LSS} that if $\om_\beta$ and $\sg_{\alpha,C}$ are weights defined in \eqref{Eq:poly weight-defn} and \eqref{Eq:Expo weight-defn}, respectively, then $\om:=\sg_{\alpha,C}/\om_\beta$ is also a weight on $G$ provided that $0<\alpha<1$ and $\beta>0$ is large enough. This, in particular, allows
us to use the weak subadditivity of $\om_\beta$ to obtain our desired result.

In this section, we will apply the approach in \cite{LSS} (mentioned above) to
show that in most cases, we get symmetric twisted Orlicz algebras if
 $|\Om|$ is determined by either weights $\sg_{\alpha,C}$ or $\rho_\gamma$
(the latter defined in  \eqref{Eq:Expo weight II-defn}).  We start with the following technical lemma which we will use to show that, similar to the case of $\sg_{\alpha,C}$, the function $\om:=\rho_\gamma/\om_\beta$ is also a weight on $G$ provided that $\gamma>0$.

\begin{lem}\label{L:difference concave functions}
Let $\beta, \gamma>0$ and $C>0$. Consider the function $p:[0,\infty) \to \R$ defined by
\begin{align}
p(x)=\frac{Cx}{(\ln (e+x))^\beta}-\gamma \ln(1+x).
\end{align}
Then there are $x_0>0$ and $M>0$ such that\\
$(i)$ $p$ is positive and increasing on $[x_0,\infty)$;\\
$(ii)$ $p'$ is positive and decreasing on $[x_0,\infty)$;\\
$(iii)$ for all $x\geq x_0$ and $y\geq 0$,
\begin{align}\label{Eq:subexpo II-poly weight subadditive}
0<p(x+y)\leq p(x)+p(y)+M.
\end{align}
\end{lem}

\begin{proof}
It is clear that $p(x)\to \infty$ as $x\to \infty$. Also we have
$$p'(x)=C(\ln(e+x))^{-\beta}-\frac{C\beta x}{e+x}(\ln(e+x))^{-1-\beta}-\frac{\gamma}{1+x}.$$
Hence, by modifying $p'(x)$ as
$$p'(x)=C(\ln(e+x))^{-\beta}(1-\frac{\beta x}{(e+x)\ln(e+x)}-\frac{\gamma
(\ln(e+x))^{\beta}}{C(1+x)})$$
and using the relations $$\lim_{x\to \infty} \frac{x}{(e+x)\ln(e+x)}=\lim_{x\to \infty} \frac{(\ln(e+x))^{\beta}}{1+x}=0,$$ it follows
that $p'(x)>0$ when $x$ is a positive large number. On the other hand, by computing the 2nd derivative of $p$, we get
$$p''(x)=\frac{C\beta (\ln(e+x))^{-2-\beta}}{(e+x)^2} q(x),$$
where
\begin{align*}
q(x)=(1+\beta)x-2e\ln (e+x)-x\ln (e+x)+\frac{\gamma (\ln(e+x))^{2+\beta}(e+x)^2}{C\beta(1+x)^2}.
\end{align*}
In particular, $p''(x)<0$ if and only if $q(x)<0$. Now if we simplify $q(x)$ as
\begin{align*}
q(x)=x[1+\beta-\frac{\ln (e+x)}{2}]+x\ln (e+x)[\frac{\gamma (\ln(e+x))^{1+\beta}(e+x)^2}{C\beta x(1+x)^2}-\frac{2e}{x}-\frac{1}{2}],
\end{align*}
 then one easily sees that  $q(x)<0$ when $x$ is a large positive number. Thus there is
 $x_0>0$ such that $p,p'>0$ and $p''<0$ on $[x_0,\infty)$. In particular, $p$ is increasing and $p'$ is decreasing on $[x_0,\infty)$.
 This proves (i) and (ii).

For (iii), fixed $y\geq 0$. Define the function $r:[x_0,\infty)\to \R$ with
$$r(x)=p(x+y)-p(x) \ \ \ \ (x\geq x_0).$$
Since $p'$ is decreasing on $[x_0,\infty)$,
 we have that $r'\leq 0$ on $[x_0,\infty)$. Thus $r(x)\leq r(x_0)$, or equivalently,
 $p(x+y)-p(x)\leq p(x_0+y)-p(x_0)$ for every $x\in [x_0,\infty)$. Hence, by applying
 the Mean Value Theorem,
 $$p(x+y)-p(x)-p(y)\leq p(x_0+y)-p(y)-p(x_0)= x_0p'(z)-p(x_0),$$
 for some $y<z<x_0+y$. The final result follows since $p'$ is continuous on $[0,\infty)$ and $\lim_{x\to \infty} p'(x)=0$
 so that it is bounded on $[0,\infty)$.
\end{proof}

\begin{prop}\label{P:Subexpo weight II-bounded-Poly weight}
Let $G$ be a compactly generated group of polynomial growth, $\beta,\gamma>0$, and
$\om_\beta$ and $\rho_\gamma$ weights \eqref{Eq:poly weight-defn} and \eqref{Eq:Expo weight II-defn}, respectively. If $\om:=\rho_\gamma/\om_\beta$, then $\om$ is a symmetric weight on $G$ satisfying GRS-condition.

%
%
\end{prop}

\begin{proof}
It is clear that $\om$ is symmetric and satisfies GRS-condition. Hence it only remains to show that $\om$
is a weight, i.e. there is constant $K>0$ such that
\begin{align}\label{Eq:bound-poly subexpo weight}
\om(st)\leq K\om(s)\om(t)  \ \ \ \ (s,t\in G).
\end{align}
Suppose that $p$ be the function defined in Lemma \ref{L:difference concave functions}.
Then clearly 
$$\om(s)=e^{p(\tau(s))} \ \ \ (s\in G).$$
Let $x_0>0$ and $M>0$ be the constants obtained in Lemma \ref{L:difference concave functions}. Take $s,t\in G$. We consider three cases:

{\it Case I:} $\max\{\tau(s),\tau(t)\}<x_0$. Then $\tau(st)\leq \tau(s)+\tau(t)<2x_0$. Hence we
can put $K_1=\max \{ e^{p(x)-p(y)-p(z)}: x,y,z \in [0,2x_0]\}$ to obtain \eqref{Eq:bound-poly subexpo weight}.

{\it Case II:} $\max\{\tau(s),\tau(t)\}\geq x_0$ and $\tau(st)< x_0$. In this case, by \eqref{Eq:subexpo II-poly weight subadditive},
$$0<p(\tau(s)+\tau(t)) \leq p(\tau(s))+p(\tau(t))+M.$$
Hence
$$e^{p(\tau(st))}\leq K_2e^{p(\tau(s))}e^{p(\tau(t))},$$
where $K_2 =e^{M}\max \{ e^{p(x)}: x \in [0,x_0]\}$.

{\it Case III:} $\max\{\tau(s),\tau(t)\}\geq x_0$ and $\tau(st)\geq x_0$. Then, since $p$ is increasing
on $[x_0,\infty)$ (Lemma \ref{L:difference concave functions}(i)), we can apply again \eqref{Eq:subexpo II-poly weight subadditive} to get
$$e^{p(\tau(st))}\leq e^{p(\tau(s))+p(\tau(t))}\leq K_3e^{p(\tau(s))}e^{p(\tau(t))},$$
where $K_3=e^{M}$.

The final result follows if we put $K=\max \{K_1,K_2,K_3\}$.
\end{proof}

We are now ready to state the main result of this section whose proof relies on
Theorem \ref{T:twisted Orlicz alg-weak add-symm}.

\begin{thm}\label{T:twisted Orlicz alg-poly and exp weight}
Let $\Om_\T\in \ZTb$, let $\tc$ be the twisted convolution coming from $\Om_\T$. Then:\\
$(i)$ If  $\sg_{\alpha,C}$ is the subexponential weight defined in \eqref{Eq:Expo weight-defn} with $0<\alpha <1$, then $(L_{\sg_{\alpha,C}}^\Phi(G), \tc)$ is a symmetric twisted Orlicz algebra.\\
$(ii)$ If $\rho_{\gamma,C}$ is the subexponential weight defined
in \eqref{Eq:Expo weight II-defn} with $\gamma>0$, then $(L_{\rho_{\gamma,C}}^\Phi(G), \tc)$ is a symmetric twisted Orlicz algebra.
\end{thm}

\begin{proof}
(i) Fix $0<\alpha<1$ and $C>0$. 
Suppose that $\beta$ is a positive number satisfying
$$\beta > \max \{1,\frac{6}{C\alpha(1-\alpha)},d(G)\},$$ where $d(G)$ is the degree of the growth of $G$.
As it is shown in the proof of Corollary \ref{C:twisted Orlicz alg-poly weight}, $1/\om_\beta \in L^\Psi(G)$. Moreover, by \cite[Theorem 3.3]{LSS}, there is a positive number $M$ (depending only on $\alpha$, $\beta$, and $C$) such that
\begin{eqnarray}\label{Eq:subexpo mod poly weight-II}
\frac{\sg_{\alpha,C}(st)}{\sg_{\alpha,C}(s)\sg_{\alpha,C}(t)}\leq
\frac{M\om_\beta(st)}{\om_\beta(s)\om_\beta(t)} \ \ \ (s,t\in G).
\end{eqnarray}
But $\om_\beta$ is weakly subadditive. In fact,
$$\om_\beta(st) \leq 2^\beta \Big(\om_\beta(s)+\om_\beta(t)\Big) \ \ (s,t\in G).$$
Therefore
$$\frac{\sg_{\alpha,C}(st)}{\sg_{\alpha,C}(s)\sg_{\alpha,C}(t)}\leq
\frac{2^\beta M}{\om_\beta(s)}+\frac{2^\beta M}{\om_\beta(t)} \ \ \ (s,t\in G).$$
Hence if we put
$$u=\frac{2^\beta M}{\om_\beta},$$
then $u\in L^\Psi(G)$ with
$$\frac{\sg_{\alpha,C}(st)}{\sg_{\alpha,C}(s)\sg_{\alpha,C}(t)}\leq
u(s)+u(t) \ \ \ (s,t\in G).$$
Now suppose that $\Om\in \Zbw$ is a 2-cocycle for which $\sg_{\alpha,C}$ is a weight associated to
$|\Om|$. Thus the preceding inequality is nothing but
$$|\Om(s,t)| \leq u(s)+u(t) \ \ \ (s,t\in G).$$
This is a particular case of the relation \eqref{Eq:2-cocycle bdd sum}, and so, by Theorem \ref{T:twisted Orlicz alg}, $(L^\Phi(G),\tw)$ is an algebra. In fact, we have that (see the relation \eqref{Eq:twisted convolution-diff norm relation} in Theorem \ref{T:twisted Orlicz alg})
for every $f,g\in L^\Phi(G)$
$$\|f\tw g\|_{\Phi} \leq \|fu\|_{1} \|g\|_{\Phi}+\|f\|_{\Phi}\|gu\|_{1}\leq C\|f\|_\Phi \|g\|_\Phi,$$
where $C=2 N_\Psi(u)$ (the last inequality follows from Holder's inequality \eqref{Eq:Holder inequality}).
Alternatively, if we apply Lemma \ref{L:twist covn-weight conv} and use the equivalent weighted reformulation, we get
\begin{eqnarray*}
\|f\tc g\|_{\Phi,\sg_{\alpha, C}} &\leq& C\|f\sg_{\alpha, C}\|_{\Phi}\|g\sg_{\alpha, C}\|_{\Phi} \\
&=& C\|f\|_{\Phi,\sg_{\alpha, C} }\|g\|_{\Phi,\sg_{\alpha, C}},
\end{eqnarray*}
for every $f,g\in L^\Phi_{\sg,\alpha}(G)$. That is $(L^\Phi_{\sg_{\alpha, C}}(G),\tc)$ is an algebra. Moreover,
by putting $\sg:=\sg_{\alpha,C}$, $\om:=\om_\beta$, and $\rho:=\sg_{\alpha,C}/\om_\beta$ and applying Theorem \ref{T:twisted Orlicz alg-weak add-symm}, we have that
$\rho$ is a symmetric weight on $G$ (see the proof of Theorem \ref{T:twisted Orlicz alg-weak add-symm} and, in particular, compare \eqref{Eq:2-cocycle bdd by subadditive weight} with \eqref{Eq:subexpo mod poly weight-II}) and $(L_{\sg_{\alpha,C}}^\Phi(G), \tc)$ is a differential $*$-subalgebra of $(L^1_\rho(G),\tc)$.
On the other hand, since both $\sg_{\alpha,C}$ and $\om_\beta$ satisfy the GRS-condition, then
so does $\rho$. Therefore, by Theorem \ref{T:symm-twisted group alg}, $(L^1_\rho(G),\tc)$
is symmetric, and so, $(L_{\sg_{\alpha,C}}^\Phi(G), \tc)$ is symmetric, once again from Theorem
\ref{T:twisted Orlicz alg-weak add-symm}.\\
(ii) In line of Proposition \ref{P:Subexpo weight II-bounded-Poly weight}, the proof is similar to part (ii).
\end{proof}

\subsection{Functional calculus, Wiener property, and minimal ideals}
\label{S:Functional calculus, Wiener property, and minimal ideals}

Let $G$ be a compactly generated group of polynomial growth, and let $\om$ be a symmetric weight on $G$.
In \cite[Sections 4-8]{KM}, the authors investigated other important and related properties of the Banach $*$-algebra $(L^p_\om(G),*)$ such as existence
of $C^\infty$-functional calculus on compactly supported self-adjoint elements of $(L^p_\om(G),*)$, regularity, having
(weak) Wiener property and existence of the minimal ideal associated to a given closed subset of the dual of $(L^p_\om(G),*)$.
Their approach is to apply what is known about these properties as well as the techniques in the case of $(L^1_\om(G),*)$
to their setting. In fact, when one examines carefully the arguments presented in \cite[Sections 4-8]{KM}, one realizes that to obtain their results,
it suffices that following holds for the Banach $*$-algebra $(L^p_\om(G),*)$:\\

$(i)$ $(L^p_\om(G),*)$ is a symmetric subalgebra of $(L^1(G),*)$;

$(ii)$ $(L^p_\om(G),*)$ is an essential Banach $(L^1_\om(G),*)$-bimodule;

$(iii)$ $(L^1_\om(G),*)$ has a bounded approximate identity consists of compactly supported selfadjoint elements;

$(iv)$ there is an appropriate $C^\infty$-functional calculus on compactly supported self-adjoint elements of $(L^1_\om(G),*)$.\\

Now suppose that $\Om\in \Zbw$ and $\om$ is a symmetric weight associated to $|\Om|$.
If $\om$ satisfies the GRS-condition, then, by Theorem \ref{T:symm-twisted group alg}, both $(L^1(G),\tc)$ and $(L^1_\om(G),\tc)$ are symmetric,
where $\tc$ is the twisted convolution coming from $\Om_\T$.
If, in addition, we have that $(L^\Phi(G), \tw)$ is a symmetric twisted Orlicz algebra, then, by Lemma \ref{L:Orlicz L1 mod} and
Theorem \ref{T:twisted Orlicz alg}, $(\Sm^\Phi_\om(G),\tc)$ satisfies the assumptions
(i) and (ii) above, where here we are replacing $(L^p_\om(G),*)$ with
$(\Sm^\Phi_\om(G),\tc)$ and the (untwisted) convolution $*$ with the twisted convolution $\tc$.
Furthermore, by Lemma \ref{L:twist covn-weight conv}, Proposition \ref{P:symm-twisted group alg} and \cite{DLM}, $(L^1_\om(G),\tc)$ satisfies the assumption (iii) and
(iv) above.
Hence all the results presented in \cite[Sections 4-8]{KM} remains valid in our setting if we replace $(L^p_\om(G),*)$ with
$(\Sm^\Phi_\om(G),\tc)$. In particular, this applies to all the cases  considered in Sections \ref{S:Symm-twisted Orlicz alg-PG groups} and \ref{S:subexp weights-symm Orlicz space}.
Since one does not need any further major argument to achieve this (beyond what we explained above), we do not present anything further and just
refer the interested reader to \cite[Sections 4-8]{KM} and \cite{DLM} (see also \cite{FGLLM}).


\section{Acknowledgements}
This work was initiated and completed while the second named author visited Istanbul University under a fellowship from Tubitak for visiting scientists and scientists on sabbatical leave. The authors would like to acknowledge the generous support of Tubitak. The first named author would like to express his deep gratitude toward his host, Serap
\"{O}ztop, whose hospitality and support made his year-long stay at Istanbul University very enjoyable.
The second named author also benefitted from correspondence with M. Leinert, whom he is deeply grateful, for clarification about the central extension of locally compact groups presented in Section \ref{S:Symmetry of twisted group algebras} as well as providing the reference \cite{L}. The authors would like to thank the referee for carefully reading the paper as well as providing the reference \cite{JW}.

\end{document}